%% file: main.tex
\documentclass[submission,copyright,creativecommons]{eptcs}
\providecommand{\event}{TARK 2023} 
\input{macros.tex}
\usepackage{enumitem}
\usepackage[all]{xy}
\usepackage{stmaryrd}
\usepackage{iftex}
\usepackage{mathrsfs}
\usepackage{amsmath}
\usepackage{amsthm}
\usepackage{amssymb}
\usepackage{mathtools}
\usepackage{longtable}
\usepackage{xcolor}
\ifpdf
  \usepackage{underscore}         
  \usepackage[T1]{fontenc}        
\else
  \usepackage{breakurl}           
\fi
\newtheorem{theorem}{Theorem}
\newtheorem{proposition}{Proposition}

\theoremstyle{definition}
\newtheorem{definition}{Definition}
\theoremstyle{remark}
\newtheorem{remark}{Remark}
\newtheorem{example}{Running example, part}
\newtheorem{convention}{Convention}
\title{Presumptive Reasoning in a Paraconsistent Setting}
\author{Sabine Frittella\qquad\qquad Daniil Kozhemiachenko
\institute{INSA Centre Val de Loire, Univ.\ Orl\'{e}ans, LIFO EA 4022, France\thanks{The research of Sabine Frittella and Daniil Kozhemiachenko was funded by the grant ANR JCJC 2019, project PRELAP (ANR-19-CE48-0006).}}
\email{\{sabine.frittella,daniil.kozhemiachenko\}@insa-cvl.fr}
\and
Bart Verheij
\institute{Bernoulli Institute, Rijksuniversiteit Groningen\\Groningen, the Netherlands\thanks{Bart Verheij participates in the Hybrid Intelligence Center, a 10-year programme funded by the Dutch Ministry of Education, Culture and Science through the Netherlands Organisation for Scientific Research, https://hybrid-intelligence-centre.nl.}}
\email{bart.verheij@rug.nl}
}
\def\titlerunning{Presumptive Reasoning in a Paraconsistent Setting}
\def\authorrunning{S.\ Frittella, D.\ Kozhemiachenko \& B. Verheij}
\begin{document}
\maketitle
\begin{abstract}
We explore presumptive reasoning in the paraconsistent case. Specifically, we provide semantics for non-trivial reasoning with presumptive arguments with contradictory assumptions or conclusions. We adapt the case models proposed by Verheij~\cite{Verheij2016,Verheij2017} and define the paraconsistent analogues of the three types of validity defined therein: coherent, presumptively valid, and conclusive ones. To formalise the reasoning, we define case models that use $\BDtriangle$, an expansion of the Belnap--Dunn logic with the Baaz Delta operator. We also show how to recover presumptive reasoning in the original, classical context from our paraconsistent version of case models. Finally, we construct a~two-layered logic over $\BDtriangle$ and $\biG$ (an expansion of G\"{o}del logic with a coimplication $\Yleft$ or $\triangle$) and obtain a faithful translation of presumptive arguments into formulas.
\end{abstract}
\section{Introduction\label{sec:introduction}}
When arguing for a given statement, it can happen that a person uses contradictory assumptions. From the classical standpoint, every statement trivially follows from a contradiction. This, however, is counter-intuitive as an agent may not be willing to accept a completely arbitrary statement just because their premises contain a contradiction.

In general, an argument from $\phi$ to $\chi$ (written formally as $\langle\phi,\chi\rangle$) can be either \emph{deductive} (when $\phi$~\emph{entails} $\chi$) or \emph{presumptive} (otherwise). I.e., to verify the correctness of a deductive argument, it suffices to utilise purely logical means while establishing the correctness (acceptability) of a presumptive one requires an extra-logical framework. Thus, from \emph{the classical standpoint}, every argument from a contradictory premise is \emph{deductive}. Hence, if one wants to formalise \emph{non-trivial} presumptive reasoning from contradictory premises, one has to use a \emph{paraconsistent} logic, i.e., a logic where the explosion principle $p,\neg p\models q$ is not valid.

\paragraph{Dung's argumentative semantics vs case models}
An influential approach to the formalisation of argumentation focuses on argument attack~\cite{Dung1995}. The main idea is to represent the argumentative framework as a directed graph where $\mathcal{A}\rightarrow\mathcal{B}$ is interpreted as ‘argument $\mathcal{A}$ attacks $\mathcal{B}$’ (here, arguments are treated as unified statements, and premises and conclusions are not singled out). Then, $\mathcal{A}$ is \emph{acceptable} if it responds to every attack (or, formally, if $\mathcal{A}\rightarrow\mathcal{B}$ for every $\mathcal{B}$ s.t.\ $\mathcal{B}\rightarrow\mathcal{A}$). An argument's correctness depends on the argumentation semantics choice. 

However, the connection of Dung's approach to standard logical semantics may not be straightforward. In addition, the support of arguments is abstracted. Both issues have been addressed in several ways (cf., e.g.~\cite{BondarenkoDungKowalskiToni1997,BesnardHunter2001,GarciaSimari2004,Prakken2010}). One such alternative to Dung's approach was proposed in~\cite{Verheij2016} and further developed in~\cite{Verheij2017}.
In these works, the interpretation of presumptive arguments was given by means of \emph{case models}: sets of classically incompatible satisfiable propositional formulas called ‘cases’ (whence the name) with a~preference relation defined thereon. An argument in this framework has the following form: $\mathcal{A}=\langle\phi,\chi\rangle$. Here, $\phi$ is the premise, $\chi$ is the conclusion, and, in addition, $\mathcal{A}$ presents a~\emph{case}~--- $\phi\wedge\chi$. Three kinds of acceptable arguments were studied: coherent (both the premise and the conclusion are supported by at least one case), presumptively valid (both the premise and conclusion are supported by the most preferred case), and conclusive (the conclusion is supported by all cases that support the premise). Furthermore, a~representation of case models in terms of sample spaces with probability measures was devised and a correspondence between different arguments and probabilities of the corresponding events was provided. An important distinction between case models and Dung's semantics is that the validities in the former are defined via the entailment in the classical logic. Thus, one can produce a non-classical counterpart to case models by changing the underlying entailment relation.

\paragraph{Non-trivial contradictory arguments}
In both approaches discussed above, it is assumed that the acceptable arguments are not \emph{self-contradictory}. Namely, if $\phi$ is an argument in Dung's framework, it should be classically satisfiable, and if $\langle\chi,\psi\rangle$ is an argument over a case model, then both $\chi$ and $\psi$ must be classically satisfiable. This restriction is easy to explain in Dung's approach: indeed, we can claim that a contradictory argument attacks itself. In the case model setting, however, it makes sense to consider contradictory arguments and cases under the following interpretation.

Every ‘case’ in the model can be thought of as a source that gives some information regarding a given set of statements. Accordingly, the preference relation on cases shows which sources are trusted more or less. In this interpretation, it is clear that even if a source is trusted, it can provide a contradictory response to a question (e.g., a police officer testifying in court can first claim that they were unarmed while on patrol and then say ‘when I saw the suspect, I immediately drew my pistol out of the holster’) or fail to provide any information at all.

Let us now introduce the running example to illustrate the contexts that we aim to formalise.
\begin{example}[Witnessing a robbery]\label{example:robbery}
An investigator reads a report by a police officer who questioned several witnesses on a bank robbery. The relevant information is whether the perpetrator had a limp ($l$), whether they had a big bag for the robbed valuables ($b$), and whether they used the lift or the staircase~($s$) to leave the office. The report contains the following testimonies.
\begin{itemize}
\item $\mathsf{c}_1$ tells that the perpetrator indeed had a limp but cannot say anything about how they left the office; moreover, $\mathsf{c}_1$ tells that the robber put all the loot in the pockets.
\item $\mathsf{c}_2$ tells nothing about whether the perpetrator was limping and mentions that the perpetrator had a big shoulder bag; unfortunately, $\mathsf{c}_2$ is confused: they claim that they saw the robber using the lift but are also saying that ‘the lift has been out of order for half a year’.
\item $\mathsf{c}_3$ testifies that the robber had a limp but walked down the stairs; $\mathsf{c}_3$'s account is also contradictory: they describe the bag as ‘huge’ but then say that the robber put it into the pocket.
\end{itemize}
\end{example}
\noindent All witnesses gave non-classical (incomplete or contradictory) responses, whence we cannot straightforwardly represent them in the case models, nor in Dung's framework. An investigator, however, needs to draw conclusions from the accounts at hand. E.g., they might want to know how the perpetrator in fact left the building and for that, they need to know whether the perpetrator had a limp.

\paragraph{Plan of the paper}
In this paper, we adapt the case models presented in~\cite{Verheij2016,Verheij2017} to the presumptive reasoning with possibly contradictory statements. To this end, we will use $\BDtriangle$ --- the expansion of the Belnap--Dunn logic~\cite{AndersonBelnap1962,Dunn1976,Belnap2019} with a Baaz $\triangle$ operator (cf.~\cite{Baaz1996} for the original presentation of $\triangle$ in the context of fuzzy logics) originating from~\cite{SanoOmori2014}. We define the analogues of coherent, presumptively valid, and conclusive arguments and show their relations to one another. Finally, we are going to provide a logical representation of all these arguments.

The remainder of the paper is structured as follows. In Section~\ref{sec:preliminaries}, we present the syntax and semantics of $\BDtriangle$. Then, in Section~\ref{sec:casemodels}, we develop the $\BDtriangle$ case models. In Section~\ref{sec:logic}, we present a logic that formalises reasoning in $\BDtriangle$ models and construct a faithful translation of arguments to formulas. Finally, in Section~\ref{sec:conclusion}, we summarise the results and provide a plan for future research.
\section{Logical preliminaries\label{sec:preliminaries}}
The language of $\LBDtriangle$ and its $\triangle$-free fragment $\LBD$ are defined via the following grammar ($\Prop$ is a~fixed countable set of propositional variables).
\begin{align*}
\LBDtriangle\ni\phi&\coloneqq p\in\Prop\mid\neg\phi\mid(\phi\wedge\phi)\mid(\phi\vee\phi)\mid\triangle\phi
\end{align*}

There are several semantics for $\BD$ and its expansions (cf.~\cite{OmoriWansing2017} for the examples). One of the simplest is a truth-table semantics from~\cite{Belnap2019}. There, a formula can have one of the following four values corresponding to the available information regarding a statement $\phi$. $\true$ stands for ‘there is only information in support of $\phi$’; $\false$ for ‘there is only information denying $\phi$’; $\neither$ for ‘there is information neither in support nor in denial of $\phi$’; $\both$ for ‘there is information both in support and denial of $\phi$’. We also use frame semantics (cf., e.g.,~\cite{BilkovaFrittellaMajerNazari2020,KleinMajerRad2021,BilkovaFrittellaKozhemiachenkoMajerNazari2022arxiv}) as it is more convenient for the logical representation of case models.
\begin{definition}[Truth-table semantics of $\BDtriangle$]\label{def:truthtablesBDtriangle}
A $\mathbf{4}$-valuation is a map $v_\mathbf{4}:\Prop\rightarrow\{\mathbf{T},\mathbf{B},\mathbf{N},\mathbf{F}\}$ that is extended to complex formulas using the following definitions.
\scriptsize{\begin{center}
\begin{tabular}{cccc}
\begin{tabular}{c|c}
$\neg$&\\\hline
$\mathbf{T}$&$\mathbf{F}$\\
$\mathbf{B}$&$\mathbf{B}$\\
$\mathbf{N}$&$\mathbf{N}$\\
$\mathbf{F}$&$\mathbf{T}$                    
\end{tabular}
&
\begin{tabular}{c|c}
$\triangle$&\\\hline
$\mathbf{T}$&$\mathbf{T}$\\
$\mathbf{B}$&$\mathbf{T}$\\
$\mathbf{N}$&$\mathbf{F}$\\
$\mathbf{F}$&$\mathbf{F}$                
\end{tabular}
&
\begin{tabular}{c|cccc}
$\wedge$ & $\mathbf{T}$ & $\mathbf{B}$ & $\mathbf{N}$ & $\mathbf{F}$ \\\hline
$\mathbf{T}$ & $\mathbf{T}$ & $\mathbf{B}$ & $\mathbf{N}$ & $\mathbf{F}$ \\
$\mathbf{B}$ & $\mathbf{B}$ & $\mathbf{B}$ & $\mathbf{F}$ & $\mathbf{F}$ \\
$\mathbf{N}$ & $\mathbf{N}$ & $\mathbf{F}$ & $\mathbf{N}$ & $\mathbf{F}$ \\
$\mathbf{F}$ & $\mathbf{F}$ & $\mathbf{F}$ & $\mathbf{F}$ & $\mathbf{F}$
\end{tabular}
&
\begin{tabular}{c|cccc}
$\vee$ & $\mathbf{T}$ & $\mathbf{B}$ & $\mathbf{N}$ & $\mathbf{F}$ \\\hline
$\mathbf{T}$ & $\mathbf{T}$ & $\mathbf{T}$ & $\mathbf{T}$ & $\mathbf{T}$ \\
$\mathbf{B}$ & $\mathbf{T}$ & $\mathbf{B}$ & $\mathbf{T}$ & $\mathbf{B}$ \\
$\mathbf{N}$ & $\mathbf{T}$ & $\mathbf{T}$ & $\mathbf{N}$ & $\mathbf{N}$ \\
$\mathbf{F}$ & $\mathbf{T}$ & $\mathbf{B}$ & $\mathbf{N}$ & $\mathbf{F}$
\end{tabular}
\end{tabular}
\end{center}}
\end{definition}
\begin{definition}[Frame semantics for $\BDtriangle$]\label{def:BDframesemantics}
Let $\phi,\phi'\in\LBDtriangle$. For a~model $\mathfrak{M}=\langle W,v^+,v^-\rangle$ with $v^+,v^-:\Prop\rightarrow2^W$, we define notions of $w\vDash^+\phi$ and $w\vDash^-\phi$ for $w\in W$ as follows.
\begin{center}
\begin{tabular}{rclrcl}
$w\vDash^+p$& iff&$w\in v^+(p)$&$w\vDash^-p$& iff &$w\in v^-(p)$\\
$w\vDash^+\neg\phi$& iff &$w\vDash^-\phi$&$w\vDash^-\neg\phi$& iff &$w\vDash^+\phi$\\
$w\vDash^+\phi\wedge\phi'$& iff &$w\vDash^+\phi$ and $w\vDash^+\phi'$&$w\vDash^-\phi\wedge\phi'$& iff &w$\vDash^-\phi$ or $w\vDash^-\phi'$\\
$w\vDash^+\phi\vee\phi'$& iff &$w\vDash^+\phi$ or $w\vDash^+\phi'$&$w\vDash^-\phi\vee\phi'$& iff &$w\vDash^-\phi$ and $w\vDash^-\phi'$\\
$w\vDash^+\triangle\phi$& iff &$w\vDash^+\phi$&$w\vDash^-\triangle\phi$& iff &$w\nvDash^+\phi$
\end{tabular}
\end{center}
We define the \emph{positive} and \emph{negative interpretations of $\phi$} as follows: $|\phi|^+=\{w\in W\mid w\vDash^+\phi\}$; $|\phi|^-=\{w\in W\mid w\vDash^-\phi\}$.

We say that a~sequent $\phi\vdash\chi$ is \emph{satisfied on $\mathfrak{M}$} (denoted, $\mathfrak{M}\models[\phi\vdash\chi]$) iff $|\phi|^+\subseteq|\chi|^+$ and $|\chi|^-\subseteq|\phi|^-$. $\phi\vdash\chi$ is \emph{valid} iff it is satisfied on every model. In this case, we say that $\phi$ \emph{entails} $\chi$ and write $\phi\models_{\BDtriangle}\chi$.
\end{definition}

Let us make several quick observations regarding $\BDtriangle$. First, the semantical conditions of $\neg$, $\wedge$, and $\vee$ coincide with those from the classical logic. On the other hand, it is more intuitive to interpret $w\vdash^+\phi$ as ‘$w$ gives evidence for (confirms) $\phi$’ and $w\vdash^-\phi$ as ‘$w$ gives evidence against (denies) $\phi$’. Thus, $w$ confirms $\phi\wedge\phi'$ when both conjuncts are confirmed by $w$ and $w$ denies $\phi\wedge\phi'$ when at least one conjunct is denied.

The difference is that in $\BDtriangle$ the truth and falsity of a~formula \emph{are independent}. Thus, in contrast to the classical logic, neither $p\wedge\neg p\vdash q$ nor $p\vdash q\vee\neg q$ is valid. Second, the addition of $\triangle$ (read ‘it is true that’) to $\BD$ makes it weakly functionally complete (cf.~\cite{OmoriSano2015} for further details). This allows us to represent every testimony a source can give regarding $\phi$ (i.e., confirm $\phi$, contradict itself regarding $\phi$, say nothing about $\phi$ or deny $\phi$) as follows:
\begin{align*}
\mathbf{t}(\phi)&\coloneqq\triangle\phi\wedge\neg\triangle\neg\phi&\mathbf{b}(\phi)&\coloneqq\triangle\phi\wedge\triangle\neg\phi&
\mathbf{n}(\phi)&\coloneqq\neg\triangle\phi\wedge\neg\triangle\neg\phi&\mathbf{f}(\phi)&\coloneqq\neg\triangle\phi\wedge\triangle\neg\phi
\end{align*}
Note that $v_\mathbf{4}(\mathbf{x}(\phi))\!=\!\true$ if $v_\mathbf{4}(\phi)\!=\!\mathbf{X}$; and $v_\mathbf{4}(\mathbf{x}(\phi))\!=\!\false$ otherwise (with $\mathbf{x}\!\in\!\{\mathbf{t},\mathbf{b},\mathbf{n},\mathbf{f}\}$ and $\mathbf{X}\!\in\!\{\true,\both,\neither,\false\}$). Furthermore, it is possible to define $\bot$ and $\top$ s.t.\ $|\top|^+=W$, $|\top|^-=\varnothing$, $|\bot|^+=\varnothing$, $|\bot|^-=W$ as follows: $\top\coloneqq\triangle p\vee\neg\triangle p$; $\bot\coloneqq\neg\top$. $\triangle$ also allows for the internalisation of entailment: for $\mathbf{x},\mathbf{x}'\in\{\mathbf{t},\mathbf{b},\mathbf{n},\mathbf{f}\}$, the formula below is valid iff $\phi\models_{\BDtriangle}\chi$.
\begin{align*}
\phi\Rrightarrow\chi&\coloneqq\bigvee\limits_{\scriptsize{\begin{matrix}\mathbf{x}\leq_\four\mathbf{x}'\end{matrix}}}\!\!\!\!(\mathbf{x}(\phi)\wedge\mathbf{x}'(\chi))\tag{$\mathbf{f}\leq_\four\mathbf{b},\mathbf{n}\leq_\four\mathbf{t}$; $\mathbf{b}\not\leq_\four\mathbf{n}$; $\mathbf{n}\not\leq_\four\mathbf{b}$}
\end{align*}
The following property will be useful in showing how classical case models can be simulated in $\BDtriangle$.
\begin{proposition}\label{prop:BDtriangleclassicality}
Let $\phi\in\LBDtriangle$ be s.t.\ every occurrence of every variable $p$ in $\phi$ is in the scope of $\triangle$. Then for every $\BDtriangle$ model $\mathfrak{M}$ and $w\in\mathfrak{M}$, exactly one of the following holds: $w\vDash^+\phi$ and $w\nvDash^-\phi$, or $w\nvDash^+\phi$ and $w\vDash^-\phi$.
\end{proposition}
\begin{proof}
Observe that $\phi$ is constructed from the formulas of the form $\triangle\chi$ using $\neg$, $\wedge$, and $\vee$. We can now proceed by induction on $\phi$. The basis case is simple. From Definition~\ref{def:BDframesemantics}, we see that $|\triangle\chi|^+=W\setminus|\triangle\chi|^-$, whence, indeed, either $w\vDash^+\triangle\chi$ and $w\nvDash^-\triangle\chi$ or $w\nvDash^+\triangle\chi$ and $w\vDash^-\triangle\chi$. 
The cases of $\phi=\psi\vee\psi'$, $\phi=\psi\wedge\psi'$, and $\phi=\neg\psi$ can be shown by straightforward application of the induction hypothesis.
\end{proof}
\section{$\BDtriangle$ case models\label{sec:casemodels}}
In this section, we introduce the $\BDtriangle$ case models. To make the presentation clearer, let us first recall the case models from~\cite{Verheij2016,Verheij2017} and types of arguments over them that we will henceforth call \emph{classical case models} and \emph{classical arguments} since they use the classical logic as background.
\begin{definition}[Classical case models]\label{def:classicalcasemodel}
A \emph{classical case model} is a tuple $\mathfrak{C}_\CPL=\langle\mathtt{C},\preceq\rangle$ s.t.\ $\mathsf{C}$ is a finite set of pairwise incompatible classically satisfiable formulas and $\preceq$ is a total preorder on $\mathsf{C}$.
\end{definition}
\noindent The strict preorder associated with $\preceq$ is interpreted as a preference relation on the set of cases. I.e., $\phi\prec\phi'$ means that the agent prefers $\phi'$ to $\phi$ (or trusts in $\phi'$ more than in $\phi$).
\begin{definition}[Classical arguments and their types]\label{def:classicalrguments}
An \emph{argument} is a tuple $\langle\phi,\phi'\rangle$ of classical propositional formulas. The \emph{case} is the statement $\phi\wedge\phi'$, while a \emph{premise} (conclusion) is any formula $\chi$ s.t.\ $\phi\models_{\CPL}\chi$ ($\phi'\models_{\CPL}\chi$). We say that the argument is \emph{presumptive} iff $\phi\not\models_{\CPL}\phi'$.

An argument $\langle\phi,\chi\rangle$ over a classical case model $\mathfrak{C}_\CPL=\langle\mathtt{C},\preceq\rangle$ is
\begin{itemize}
\item \emph{classically coherent} iff there is $\psi\in\mathsf{C}$ s.t.\ $\psi\models_{\CPL}\phi\wedge\chi$;
\item \emph{classically conclusive} iff it is classically coherent and it holds $\psi\models_{\CPL}\phi\wedge\chi$ for every $\psi\in\mathsf{C}$ s.t.\ $\psi\models_{\CPL}\phi$;
\item \emph{classically presumptively valid} iff it is classically coherent and there is $\psi\in\mathsf{C}$ s.t.\ $\psi\models_{\CPL}\phi\wedge\psi$ and $\psi\succeq\psi'$ for every $\psi'$ s.t.\ $\psi'\models_{\CPL}\phi$.
\end{itemize}
\end{definition}

Let us now present $\BDtriangle$ case models and the counterparts to the coherent, conclusive, and presumptively valid arguments.
\begin{definition}[$\BDtriangle$ case models]\label{def:paraconsistentcasemodel}
A \emph{$\BDtriangle$ case model} is a tuple $\mathfrak{C}_{\BDtriangle}=\langle\mathsf{C},\preceq\rangle$ with $\mathsf{C}$ being a~finite set of $\LBDtriangle$ formulas s.t.\ for any $\phi,\phi'\in\mathsf{C}$, it holds that $\phi\not\models_{\BDtriangle}\bot$ and $\phi\wedge\phi'\models_{\BDtriangle}\bot$, and $\preceq$ a total preorder on $\mathsf{C}$.
\end{definition}
\begin{definition}[Arguments]\label{def:arguments}
An \emph{argument} is a tuple $\langle\phi,\phi'\rangle$ with $\phi,\phi'\in\LBDtriangle$. The \emph{case} is the statement $\phi\wedge\phi'$, while a \emph{premise} (conclusion) is any formula $\chi$ s.t.\ $\phi\models_{\BDtriangle}\chi$ ($\phi'\models_{\BDtriangle}\chi$). We say that the argument is \emph{presumptive} iff $\phi\not\models_{\BDtriangle}\phi'$.
\end{definition}
\noindent We can interpret $\psi\in\mathsf{C}$ as witnesses' testimonies. A testimony might be contradictory or omit information relevant to the case. Thus, given an argument $\langle\phi,\chi\rangle$, it makes sense to differentiate between three kinds of conclusions.
\begin{enumerate}
\item Given $\phi$, $\chi$ is claimed to be \emph{true} but nothing is said whether it is also non-false.
\item Given $\phi$, $\chi$ is claimed to be \emph{non-false} but nothing is said about whether it is true as well.
\item Given $\phi$, $\chi$ is claimed to be \emph{true and non-false}.
\end{enumerate}

Let us now recall part~\ref{example:robbery} of the running example and build a case model.
\begin{example}[Witnessing a robbery, formalisation]\label{example:coherentconclusive}
The investigator in part~\ref{example:robbery} can build the following case model $\mathfrak{C}$ (we omit the ordering for now):
\begin{align*}
\mathsf{C}&=\{\underbrace{\mathbf{t}(l)\wedge\mathbf{n}(s)\wedge\mathbf{f}(b)}_{\mathsf{c}_1},\underbrace{\mathbf{n}(l)\wedge\mathbf{b}(s)\wedge\mathbf{t}(b)}_{\mathsf{c}_2},\underbrace{\mathbf{t}(l)\wedge\mathbf{t}(s)\wedge\mathbf{b}(b)}_{\mathsf{c}_3}\}
\end{align*}
\end{example}
\noindent Using part~\ref{example:coherentconclusive} of the running example, we define the counterparts to coherent and conclusive arguments from~\cite{Verheij2016,Verheij2017}.
\begin{definition}[Coherent arguments]\label{def:+-coherence}
Let $\mathfrak{C}=\langle\mathsf{C},\preceq\rangle$.  $\langle\phi,\chi\rangle$ is
\begin{itemize}
\item \emph{negatively coherent} (denoted $\mathfrak{C}\models\phi\mapsto^-\chi$) over $\mathfrak{C}$ iff there is $\psi\in\mathsf{C}$ s.t.\ $\chi\models_{\BDtriangle}\phi\wedge\neg\triangle\neg\chi$;
\item \emph{positively coherent} (denoted $\mathfrak{C}\models\phi\mapsto^+\chi$) over $\mathfrak{C}$ iff there is $\psi\in\mathsf{C}$ s.t.\ $\chi\models_{\BDtriangle}\phi\wedge\triangle\chi$;
\item \emph{strongly coherent} (denoted $\mathfrak{C}\models\phi\mapsto^\pm\chi$) over $\mathfrak{C}$ iff there is $\psi\in\mathsf{C}$ s.t.\ $\chi\models_{\BDtriangle}\phi\wedge\mathbf{t}(\chi)$.
\end{itemize}
\end{definition}
\begin{definition}[Conclusive arguments]\label{def:+-conclusiveness}
Let $\mathfrak{C}=\langle\mathsf{C},\preceq\rangle$.  $\langle\phi,\chi\rangle$ is 
\begin{itemize}
\item \emph{negatively conclusive} over $\mathfrak{C}$ (denoted $\mathfrak{C}\!\models\!\phi\!\Rightarrow^-\!\chi$) iff it is negatively coherent and it holds that if $\chi\!\models_{\BDtriangle}\!\phi$, then $\psi\!\models_{\BDtriangle}\!\phi\!\wedge\!\neg\triangle\neg\chi$ for any $\psi\in\mathsf{C}$;
\item \emph{positively conclusive} over $\mathfrak{C}$ (denoted $\mathfrak{C}\!\models\!\phi\!\Rightarrow^+\!\chi$) iff it is positively coherent and it holds that if $\psi\!\models_{\BDtriangle}\!\phi$, then $\psi\!\models_{\BDtriangle}\!\phi\!\wedge\triangle\chi$ for any $\psi\in\mathsf{C}$;
\item \emph{strongly conclusive} over $\mathfrak{C}$ (denoted $\mathfrak{C}\!\models\!\phi\!\Rightarrow^\pm\!\chi$) iff it is strongly coherent, and it holds that if $\psi\models_{\BDtriangle}\phi$, then $\psi\models_{\BDtriangle}\phi\wedge\mathbf{t}(\chi)$ for any $\psi\in\mathsf{C}$.
\end{itemize}
\end{definition}
\begin{remark}\label{rem:coherentconclusiveinterpretation}
Let us provide an intuitive explanation of coherent and conclusive arguments. We begin with coherent arguments:
\begin{itemize}
\item for an argument to be \emph{negatively coherent}, there has to be a case that supports the premise and \emph{does not contradict the conclusion};
\item for an argument to be \emph{positively coherent}, there has to be a case that supports the premise and also \emph{supports the conclusion};
\item for an argument to be \emph{strongly coherent}, there has to be a case that supports the premise, \emph{does not contradict the conclusion}, and \emph{supports it}.
\end{itemize}

Conclusive arguments can be construed as follows:
\begin{itemize}
\item for an argument to be \emph{negatively conclusive}, no case satisfying the premise should \emph{contradict} the conclusion of a  argument;
\item for an argument to be \emph{positively conclusive}, all cases satisfying the premises of an argument should support its conclusion.
\end{itemize}

\noindent Observe that the arguments that are both positively and negatively conclusive are strongly conclusive as well. On the other hand, $\langle\phi,\chi\rangle$ can be both positively and negatively coherent but not strongly coherent if there is no case $\mathsf{c}$ s.t.\ $\mathsf{c}\models_{\BDtriangle}\mathbf{t}(\chi)$.
\end{remark}
\begin{remark}[$\BDtriangle$ and classical arguments]\label{rem:classicalvsBDtrianglearguments}
Note that while there is no classical case model over which both $\mathcal{A}=\langle\phi,\chi\rangle$ and $\mathcal{B}=\langle\phi,\neg\chi\rangle$ are classically conclusive (albeit, they can be presumptively valid), it is possible that they are both \emph{positively conclusive} (\emph{negatively conclusive}) if $\mathsf{c}\models_{\BDtriangle}\mathbf{b}(\chi)$ (resp., $\mathsf{c}\models_{\BDtriangle}\mathbf{n}(\chi)$) for every $\mathsf{c}\in\mathsf{C}$. Still, there is no $\BDtriangle$ case model over which $\mathcal{A}$ and $\mathcal{B}$ are \emph{strongly conclusive}.

In addition, it is clear that no argument of the form $\langle\phi,\neg\phi\rangle$ is \emph{classically coherent} since $\phi\wedge\neg\phi$ is classically unsatisfiable. On the other hand, $\langle s,\neg s\rangle$ is \emph{positively coherent} (by $\mathsf{c}_2$) in the model from the part~\ref{example:coherentconclusive} of the running example.

Finally, it is easy to see that every coherent deductive classical argument $\langle\phi,\chi\rangle$ (i.e., the one where $\phi\models_{\CPL}\chi$) is also classically conclusive. In the case of $\BDtriangle$ arguments, only the weaker statement holds: \emph{if $\phi\models_{\BDtriangle}\chi$ and $\langle\phi,\chi\rangle$ is positively coherent, then it is positively conclusive as well}. E.g., $p\wedge\neg p\wedge q\models_{\BDtriangle}p\wedge\neg p$ but $\mathbf{t}(p\wedge\neg p)$ always has value $\false$, whence $\langle p\wedge\neg p\wedge q,p\wedge\neg p\rangle$ can never be negatively or strongly coherent (and thus, negatively or strongly conclusive).
\end{remark}

Let us now define the $\BDtriangle$ counterparts of presumptively valid arguments.
\begin{definition}[Presumptively valid arguments]\label{def:+-presumptivevalidity}
An argument $\mathcal{A}=\langle\phi,\chi\rangle$ is:
\begin{itemize}
\item \emph{positively presumptively valid} (denoted $\mathfrak{C}\models\phi\!\rightsquigarrow^+\!\chi$) iff there is $\psi\in\mathsf{C}$ s.t.\ $\psi\models_{\BDtriangle}\phi\wedge\triangle\chi$ and $\psi\succeq\psi'$ for any $\psi'$ s.t.\ $\psi'\models\phi$;
\item \emph{negatively presumptively valid} (denoted $\mathfrak{C}\models\phi\!\rightsquigarrow^-\!\chi$) iff there is $\psi\in\mathsf{C}$ s.t.\ $\psi\models_{\BDtriangle}\phi\wedge\neg\triangle\neg\chi$ and $\psi\succeq\psi'$ for any $\psi'$ s.t.\ $\psi'\models\phi$;
\item \emph{strongly presumptively valid} (denoted $\mathfrak{C}\models\phi\!\rightsquigarrow^\pm\!\chi$) iff there is $\psi\in\mathsf{C}$ s.t.\ $\psi\models_{\BDtriangle}\phi\wedge\mathbf{t}(\chi)$ and $\psi\succeq\psi'$ for any $\psi'$ s.t.\ $\psi'\models\phi$.
\end{itemize}
\end{definition}
\begin{convention}
We will further call $\psi$ the \emph{witnessing case for $\mathcal{A}$}.
\end{convention}
\begin{remark}\label{rem:presumptivelyvalidinterpretation}
We can now explain presumptively valid arguments similarly to how we interpreted coherent and conclusive ones.
\begin{itemize}
\item An argument is \emph{negatively coherent} when there is the most preferred case that supports its premise and \emph{does not contradict the conclusion}.
\item An argument is \emph{positively coherent} when there is the most preferred case that supports \emph{both its premise and conclusion}.
\end{itemize}
\end{remark}
\begin{example}[Witnessing a robbery, preferences]\label{example:presumptivelyvalid}
We return to the model in part~\ref{example:coherentconclusive}. The investigator now wants to find out how the robber escaped from the office. It is clear that neither $\langle\top,s\rangle$ nor $\langle\top,\neg s\rangle$ is strongly conclusive. On the other hand, nobody \emph{explicitly denied} that the robber was limping, whence $\langle\top,l\rangle$ is \emph{negatively conclusive}. The investigator thinks that it is reasonable to take $l$ to be true. Unfortunately, even assuming $l$, neither $\langle l,s\rangle$ nor $\langle l,\neg s\rangle$ is conclusive.

The investigator rereads the accounts of $\mathsf{c}_1$, $\mathsf{c}_2$, and $\mathsf{c}_3$ and notices that $\mathsf{c}_3$ was the only one to follow the robber out of the office. On the other hand, $\mathsf{c}_1$ hid under the table and was sitting there during the robbery. Thus, the preference is given as $\mathsf{c}_1\prec\mathsf{c}_2\prec\mathsf{c}_3$. Hence, $\langle l,s\rangle$ is \emph{strongly presumptively valid}.
\end{example}
\begin{remark}
It is important to note that \emph{both following statements are false}:
\begin{itemize}
\item $\langle\phi,\chi\rangle$ is negatively coherent (resp., presumptively valid, conclusive) iff $\langle\phi,\neg\chi\rangle$ is positively coherent (resp., presumptively valid, conclusive);
\item $\langle\phi,\chi\rangle$ is positively coherent (resp., presumptively valid, conclusive) iff $\langle\phi,\neg\chi\rangle$ is negatively coherent (resp., presumptively valid, conclusive).
\end{itemize}
Indeed, recall part~\ref{example:coherentconclusive} of the running example. $\langle\top,l\rangle$ is negatively coherent while $\langle\top,\neg l\rangle$ is not positively coherent. $\langle b,\neg s\rangle$ is negatively presumptively valid but $\langle s,\neg b\rangle$ is positively presumptively valid but $\langle s,b\rangle$ is not negatively presumptively valid.
\end{remark}

The following statement establishes the expected relations between coherent, presumptively valid, and conclusive arguments and follows immediately from Definitions~\ref{def:+-coherence}--\ref{def:+-presumptivevalidity}.
\begin{proposition}
The diagram in Fig.~\ref{fig:argumentsdiagram} depicts the inclusions between different types of arguments.
\begin{figure}
\centering
\small{
\[\xymatrix@C=.5pc @R=.5pc{
\phi\mapsto^+\chi\ar@{}[rr]|-*[@]{\supseteq}\ar@{}[rd]|-*[@]{\supseteq}&&\phi\rightsquigarrow^+\chi\ar@{}[rr]|-*[@]{\supseteq}\ar@{}[rd]|-*[@]{\supseteq}&&\phi\Rightarrow^+\chi\ar@{}[rd]|-*[@]{\supseteq}&\\
&\phi\mapsto^\pm\chi\ar@{}[rr]|-*[@]{\supseteq}&&\phi\rightsquigarrow^\pm\chi\ar@{}[rr]|-*[@]{\supseteq}&&\phi\Rightarrow^\pm\chi\\
\phi\mapsto^-\chi\ar@{}[rr]|-*[@]{\supseteq}\ar@{}[ru]|-*[@]{\supseteq}&&\phi\rightsquigarrow^-\chi\ar@{}[rr]|-*[@]{\supseteq}\ar@{}[ru]|-*[@]{\supseteq}&&\phi\Rightarrow^-\chi\ar@{}[ru]|-*[@]{\supseteq}&
}
\]
}
\caption{Conclusive ($\Rightarrow$), presumptively valid ($\rightsquigarrow$), and coherent ($\mapsto$) arguments with same statements.}
\label{fig:argumentsdiagram}
\end{figure}
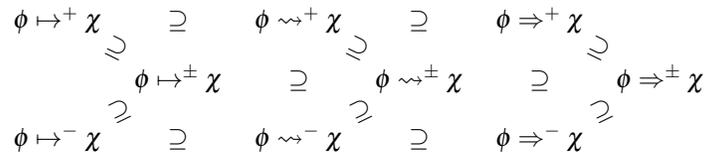
\end{proposition}

\noindent It is instructive to see that $\BDtriangle$ models allow classical presumptive reasoning if the values of formulas in the cases are classical. We define a class of $\BDtriangle$ case models ‘simulating’ the classical ones.
\begin{definition}[Quasi-classical case models]\label{def:BDtriangleclassical}
A $\BDtriangle$ case model $\mathfrak{C}=\langle\mathsf{C},\preceq\rangle$ is called \emph{quasi-classical} iff every $\chi\in\mathsf{C}$ is constructed from $\mathbf{t}(p)$'s via applications of $\neg$, $\wedge$, and $\vee$.
\end{definition}
\begin{proposition}\label{prop:samearguments}
Let $\mathfrak{C}$ be a quasi-classical $\BDtriangle$ case model and $\triangleright\in\{\mapsto,\rightsquigarrow,\Rightarrow\}$. Then $\mathfrak{C}\models\phi\triangleright^+\chi$ iff $\mathfrak{C}\models\phi\triangleright^-\chi$ iff $\mathfrak{C}\models\phi\triangleright^\pm\chi$.
\end{proposition}
\begin{proof}
We only consider the case of coherent arguments as conclusive and presumptively valid ones can be tackled similarly. It suffices to prove that positively coherent arguments and negatively coherent arguments are strongly coherent. Let $\mathfrak{C}$ be quasi-classical and $\mathfrak{C}\models\phi\mapsto^+\chi$. Then, there is $\psi\in\mathfrak{C}$ s.t.\ $\psi\models_{\BDtriangle}\phi\wedge\triangle\chi$. But then, from Definition~\ref{def:BDframesemantics} and Proposition~\ref{prop:BDtriangleclassicality}, it is clear that $\psi\models_{\BDtriangle}\mathbf{t}(\chi)$. Likewise, let $\mathfrak{C}\models\phi\mapsto^-\chi$, and, accordingly, $\psi\models_{\BDtriangle}\phi\wedge\neg\triangle\neg\chi$. Again, using Definition~\ref{def:BDframesemantics} and Proposition~\ref{prop:BDtriangleclassicality}, we have $\psi\models_{\BDtriangle}\mathbf{t}(\chi)$. The result now follows.
\end{proof}

\noindent We finish the section by showing how to build a $\BDtriangle$ counterpart of a classical case model
that preserves all arguments.
\begin{definition}\label{def:BDtrianglecounterparts}
Let $\mathfrak{C}=\langle\mathsf{C},\preceq\rangle$ be a classical case model s.t.\ all formulas in $\mathsf{C}$ are over $\{\neg,\wedge,\vee\}$. In addition, for $\phi\in\LBD$, denote $\phi^\mathbf{t}$ the result of substitution of every variable $p$ occurring in $\phi$ with $\mathbf{t}(p)$.

The \emph{$\BDtriangle$ counterpart} of $\mathfrak{C}$ is $\mathfrak{C}_{\BDtriangle}=\langle\mathsf{C}_{\BDtriangle},\preceq_{\BDtriangle}\rangle$ with $\mathsf{C}_{\BDtriangle}=\{\chi^\mathbf{t}:\chi\in\mathsf{C}\}$ and $\chi^\mathbf{t}\preceq_{\BDtriangle}{\chi'}^\mathbf{t}$ iff $\chi\preceq_{\BD}\chi'$.
\end{definition}
\begin{theorem}\label{theorem:BDtrianglecounterparts}
Let $\mathfrak{C}=\langle\mathsf{C},\preceq\rangle$ be a classical case model s.t.\ all formulas in $\mathsf{C}$ are over $\{\neg,\wedge,\vee\}$.
\begin{enumerate}
\item The $\BDtriangle$ counterpart $\mathfrak{C}_{\BDtriangle}$ of $\mathfrak{C}$ is quasi-classical.
\item $\langle\phi,\chi\rangle$ is coherent (resp., presumptively valid, conclusive) in $\mathfrak{C}$ iff $\langle\phi^\mathbf{t},\chi^\mathbf{t}\rangle$ is strongly coherent (resp., strongly presumptively valid, strongly conclusive) in $\mathfrak{C}_{\BDtriangle}$.
\end{enumerate}
\end{theorem}
\begin{proof}
1.\ holds by Definitions~\ref{def:classicalcasemodel} and~\ref{def:BDtrianglecounterparts}. Let us now consider 2. We will only tackle the case of presumptively valid arguments since coherent and conclusive ones can be dealt with similarly.

Let $\mathfrak{C}=\langle\mathsf{C},\preceq\rangle$ be a classical case model and $\langle\phi,\chi\rangle$ presumptively valid on $\mathfrak{C}$. Then, there is $\psi\in\mathsf{C}$ s.t.\ $\psi\models_{\CPL}\phi\wedge\chi$ and $\psi\succeq\psi'$ for every $\psi'\in\mathsf{C}$ s.t.\ $\psi'\models_{\CPL}\phi$. Now observe from Definition~\ref{def:truthtablesBDtriangle} that $\neg$, $\wedge$, and $\vee$ behave classically on $\true$ and $\false$. Thus, it is clear that $\tau\models_{\CPL}\tau'$ iff $\tau^\mathbf{t}\models_{\BDtriangle}{\tau'}^\mathbf{t}$ for every $\tau,\tau'\in\LBD$. It now follows that $\psi^\mathbf{t}\models_{\BDtriangle}\phi^\mathbf{t}\wedge\mathbf{t}(\chi^\mathbf{t})$ and $\psi^\mathbf{t}\succeq_{\BDtriangle}{\psi'}^\mathbf{t}$ for every ${\psi'}^\mathbf{t}\models_{\BDtriangle}\phi^\mathbf{t}$, as required. Conversely, let $\langle\phi,\chi\rangle$ be not presumptively valid on $\mathfrak{C}$. Then, there is no $\psi\in\mathsf{C}$ s.t.\ $\psi\models_{\CPL}\phi\wedge\chi$ and $\psi\succeq\psi'$ for every $\psi'\in\mathsf{C}$ s.t.\ $\psi'\models_{\CPL}\phi$. Again, from Definition~\ref{def:truthtablesBDtriangle}, it follows that there is no $\psi^\mathbf{t}$ s.t.\ $\psi^\mathbf{t}\models_{\BDtriangle}\phi^\mathbf{t}\wedge\mathbf{t}(\chi^\mathbf{t})$ and $\psi^\mathbf{t}\succeq_{\BDtriangle}{\psi'}^\mathbf{t}$ for every ${\psi'}^\mathbf{t}\models_{\BDtriangle}\phi^\mathbf{t}$.
\end{proof}
\section{A two-layered logic for case models\label{sec:logic}}
Conclusive and presumptively valid arguments on \emph{classical} case models can be represented in terms of conditional probabilities~\cite{Verheij2016,Verheij2017}. In this section, we provide a representation of $\BDtriangle$ models and arguments on them in terms of a paraconsistent two-layered logic.

Two-layered logics form a class of formalisms designed to reason about uncertainty: their languages consist of \emph{inner-layer} formulas that describe events and \emph{outer-layer} formulas composed of \emph{modal atoms} of the form $\mathsf{M}\phi$ ($\phi$ being an inner-layer formula and $\mathsf{M}$ the modality interpreted as a measure on the set of events). There are two-layered logics formalising reasoning with classical probabilities~\cite{BaldiCintulaNoguera2020} and their paraconsistent counterparts~\cite{BilkovaFrittellaKozhemiachenkoMajerNazari2022arxiv} presented in~\cite{KleinMajerRad2021}.\footnote{Note that~\cite{KleinMajerRad2021} is not the only paraconsistent interpretation of probabilities: alternative approaches can be found. e.g., in~\cite{Mares1997,Dunn2010,Bueno-SolerCarnielli2016,RodriguesBueno-SolerCarnielli2021}.} Furthermore, there are two-layered logics formalising paraconsistent reasoning with belief and plausibility functions~\cite{BilkovaFrittellaKozhemiachenkoMajerNazari2022arxiv}.

These papers usually study \emph{quantitative} representations of uncertainty. On the other hand, case models provide a \emph{qualitative one} via their preference relations. This shows a degree of affinity between case models and representations of different uncertainty measures by means of total preorders as studied in~\cite{KraftPrattSeidenberg1959} (for the case of probabilities) and~\cite{WongYaoBollmannBurger1991,WongYaoBollmann1992} (belief functions). In~\cite{BilkovaFrittellaKozhemiachenkoMajer2023IJAR}, two-layered logics formalising reasoning with the qualitative counterparts of belief functions and probabilities were presented.

In this section, we present a two-layered logic $\QGBDtriangle$ which is a modification of $\QG$ from~\cite{BilkovaFrittellaKozhemiachenkoMajer2023IJAR}. The inner layer of $\triangleMCB$ is $\BDtriangle$, the outer one is $\biG$ --- an expansion of G\"{o}del logic (cf., e.g.,~\cite{Hajek1998}) with a coimplication $\coimplies$ or the Baaz Delta operator $\triangle$. To connect the layers, we use $\Be$ (with $\Be\phi$ read as ‘the agent believes in $\phi$’). Since we do not impose any restrictions on $\preceq$ in case models, we are interpreting $\Be$ as a \emph{capacity} on the set of events $W$, i.e., via a map $\mu:2^W\rightarrow[0,1]$ which is monotone w.r.t.\ $\subseteq$ with $\mu(W)=1$ and $\mu(\varnothing)=0$. The main goal of the paper is to establish a correspondence between case models and $\QGBDtriangle$ models as well as to show how given an argument $\langle\phi,\phi'\rangle$ to construct a $\QGBDtriangle$ formula that is true in the corresponding model iff the argument is (positively, negatively, or strongly) coherent, conclusive, or presumptively valid.

Let us now recall $\biG$.
\begin{definition}\label{def:biGalgebra}
The bi-G\"{o}del algebra $[0,1]_{\mathsf{G}}=\langle[0,1],0,1,\wedge_\mathsf{G},\vee_\mathsf{G},\rightarrow_{\mathsf{G}},\coimplies,\sim_\mathsf{G},\triangle_\mathsf{G}\rangle$ is defined as follows: for all $a,b\in[0,1]$, $\wedge_\mathsf{G}$ and $\vee_\mathsf{G}$ are given by $a\wedge_\mathsf{G}b\coloneqq\min(a,b)$, $a\vee_\mathsf{G}b\coloneqq\max(a,b)$. The remaining operations are defined below:
\begin{align*}
a\rightarrow_\mathsf{G}b&=
\begin{cases}
1,\text{ if }a\leq b\\
b\text{ else}
\end{cases}
&
a\coimplies_\mathsf{G}b&=
\begin{cases}
0,\text{ if }a\leq b\\
a\text{ else}
\end{cases}
&
{\sim_\mathsf{G}}a&=
\begin{cases}
0,\text{ if }a>0\\
1\text{ else}
\end{cases}
&
\triangle_\mathsf{G}a&=
\begin{cases}
0,\text{ if }a<1\\
1\text{ else}
\end{cases}
\end{align*}
\end{definition}
\begin{remark}
Note that constants $\top$ and $\bot$ are definable as, respectively, $p\rightarrow p$ and $p\coimplies p$, and that $\triangle$ and $\coimplies$ are interdefinable as follows: $\triangle\phi\coloneqq\top\coimplies(\top\coimplies\phi)$, $\phi\coimplies\phi'\coloneqq\phi\wedge{\sim}\triangle(\phi\rightarrow\phi')$.
\end{remark}
\begin{definition}[Language and semantics of $\biG$]
\label{def:semantics:G2}
We fix a~countable set $\Prop$ of propositional variables and consider the following language.
\begin{align*}
\LbiG\ni\phi&\coloneqq p\in\Prop\mid{\sim}\phi\mid(\phi\wedge\phi)\mid(\phi\vee\phi)\mid(\phi\rightarrow\phi)\mid(\phi\Yleft\phi)\mid\triangle\phi
\end{align*}

Let $e:\mathsf{Prop}\rightarrow [0,1]$. For the complex formulas, we define $e(\phi\circ\phi')=e(\phi)\circ_\mathsf{G}e(\phi')$.

Finally, let $\Gamma\cup\{\phi\}\subseteq\LbiG$. We define: $\Gamma\models_{\biG}\phi$ iff $\forall e:\inf\{e(\psi):\psi\in\Gamma\}\leq e(\phi)$.
\end{definition}
\noindent Using $\biG$, we can define $\QGBDtriangle$ as follows.
\begin{definition}\label{def:triangleMCB}
The language of $\QGBDtriangle$ is defined via the following grammar: $\LQGBDtriangle\!\ni\!\alpha\coloneqq\Be\phi\mid\alpha\circ\alpha$ $(\circ\in\{{\sim},\wedge,\vee,\rightarrow,\coimplies,\triangle\},\phi\in\LBDtriangle)$.
A $\QGBDtriangle$ model is a~tuple $\mathscr{M}=\langle W,v^+,v^-,\mu,e\rangle$ with $\langle W,v^+,v^-\rangle$ being a~$\BDtriangle$ model (cf.~Definition~\ref{def:BDframesemantics}), $\mu:2^W\rightarrow[0,1]$ being a capacity. Semantic conditions of $\LQGBDtriangle$ formulas are as follows: $e(\Be\phi)=\mu(|\phi|^+)$ for modal atoms; the values of complex formulas are computed according to Definition~\ref{def:semantics:G2}.

For a given model $\mathscr{M}$, we write $\mathscr{M}\models\alpha$ to designate $e(\alpha)=1$. For a~frame $\mathbb{F}=\langle W,\pi\rangle$ on a~$\QGBDtriangle$ model $\mathscr{M}$, we say that $\alpha\in\LtriangleMCB$ is valid on $\mathbb{F}$ ($\mathbb{F}\models\alpha$) iff $e(\alpha)\!=\!1$ for every $e$ on~$\mathbb{F}$. Finally, for $\Psi\cup\{\alpha\}\subseteq\LQGBDtriangle$, we define the same entailment relation as in Definition~\ref{def:semantics:G2}.
\end{definition}

Let us now establish the correspondence results for coherent, conclusive, and presumptively valid arguments. To do this, we define a class of \emph{$\mu$-counterparts} for every $\BDtriangle$ model.
\begin{definition}[$\mu$-counterparts]\label{def:mucounerparts}
Let $\mathfrak{C}=\langle\mathsf{C},\preceq\rangle$ be a $\BDtriangle$ case model and $\mathsf{C}=\{\mathsf{c}_1,\ldots,\mathsf{c}_n\}$. Its $\mu$-counterpart is a $\QGBDtriangle$-model $\mathscr{M}_{\mathfrak{C}}=\langle\{w_1,\ldots,w_n\},v^+,v^-,\mu_{\preceq},e\rangle$ for which the following holds.
\begin{enumerate}
\item For every $\mathsf{c}_i\in\mathsf{C}$ and every $\phi$, if $\mathsf{c}_i\models_{\BDtriangle}\phi$ ($\mathsf{c}_i\models_{\BDtriangle}\neg\phi$), then $w_i\vDash^+\phi$ ($w_i\vDash^-\phi$).
\item For every $\mathsf{c}_i,\mathsf{c}_j\in\mathsf{C}$, $\mathsf{c}_i\preceq\mathsf{c}_j$ iff $\mu_{\preceq}(\{w_i\})\leq\mu_{\preceq}(\{w_j\})$.
\item For every $\mathsf{c}_i\in\mathsf{C}$, $\mu_{\preceq}(\{\mathsf{c}_i\})>0$.
\end{enumerate}
\end{definition}

\begin{theorem}\label{theorem:coherentrepresenation}
Let $\mathfrak{C}=\langle\mathsf{C},\preceq\rangle$ be a $\BDtriangle$ case model and $\mathscr{M}_{\mathfrak{C}}$ its $\mu$-counterpart. Then the following holds.
\begin{enumerate}
\item $\mathfrak{C}\models\phi\mapsto^+\phi'$ iff $\mathscr{M}_{\mathfrak{C}}\models{\sim\sim}\Be(\phi\wedge\triangle\phi')$.
\item $\mathfrak{C}\models\phi\mapsto^-\phi'$ iff $\mathscr{M}_{\mathfrak{C}}\models{\sim\sim}\Be(\phi\wedge\neg\triangle\neg\phi')$.
\item $\mathfrak{C}\models\phi\mapsto^\pm\phi'$ iff $\mathscr{M}_{\mathfrak{C}}\models{\sim\sim}\Be(\phi\wedge\mathbf{t}(\phi'))$.
\end{enumerate}
\end{theorem}
\begin{proof}
We consider 2. Other cases can be proved in the same way. Let $\mathfrak{C}\models\phi\mapsto^-\phi'$. Then, there is $\mathsf{c}_i\in\mathfrak{C}$ s.t.\ $\mathsf{c}_i\models_{\BDtriangle}\phi\wedge\neg\triangle\neg\phi'$, whence $w_i\vDash^+\phi\wedge\neg\triangle\neg\phi'$ and $\mu(|\phi\wedge\neg\triangle\neg\phi'|^+)>0$. Thus, $e(\Be(\phi\wedge\neg\triangle\neg\phi'))>0$ and $\mathscr{M}_\mathfrak{C}\models{\sim\sim}\Be(\phi\wedge\neg\triangle\neg\phi')$, as required. Conversely, let $\mathfrak{C}\not\models\phi\mapsto^-\phi'$. Then, for every $\mathsf{c}_i\in\mathfrak{C}$, $\mathsf{c}_i\not\models_{\BDtriangle}\phi\wedge\neg\triangle\neg\phi'$, whence there is no $w_i$ s.t.\ $w_i\vDash^+\phi\wedge\neg\triangle\neg\phi'$. Hence, $|\phi\wedge\neg\triangle\neg\phi'|^+=\varnothing$ and $\mu(|\phi\wedge\neg\triangle\neg\phi'|^+)=0$. Thus, 
${\sim\sim}e(\Be(|\phi\wedge\neg\triangle\neg\phi'|^+))=0$, as required.
\end{proof}

\noindent Observe from Definitions~\ref{def:+-conclusiveness} and~\ref{def:+-presumptivevalidity} that the classes of strongly conclusive (presumptively valid) arguments on the one hand and both positively and negatively conclusive (presumptively valid) arguments on the other hand coincide. Thus, it suffices to provide representation for positively and negatively conclusive (presumptively valid) arguments only.
\begin{theorem}\label{theorem:conclusiverepresentation}
Let $\mathfrak{C}=\langle\mathsf{C},\preceq\rangle$ be a $\BDtriangle$ case model and $\mathscr{M}_{\mathfrak{C}}$ its $\mu$-counterpart. Then the following holds.
\begin{enumerate}
\item $\mathfrak{C}\models\phi\Rightarrow^+\phi'$ iff $\mathscr{M}_\mathfrak{C}\models{\sim}\Be(\phi\wedge\neg\triangle\phi')\wedge{\sim\sim}\Be(\phi\wedge\triangle\phi')$.
\item $\mathfrak{C}\models\phi\Rightarrow^-\phi'$ iff $\mathscr{M}_\mathfrak{C}\models{\sim}\Be(\phi\wedge\triangle\neg\phi')\wedge{\sim\sim}\Be(\phi\wedge\neg\triangle\neg\phi')$.
\end{enumerate}
\end{theorem}
\begin{proof}
Again, for the sake of brevity, we consider only 1. We let $\mathfrak{C}\models\phi\Rightarrow^+\phi'$. Then, $\langle\phi,\phi'\rangle$ is positively coherent on $\mathfrak{C}$ and thus (by Theorem~\ref{theorem:coherentrepresenation}), $\mathscr{M}_{\mathfrak{C}}\models{\sim\sim}\Be(\phi\wedge\triangle\phi')$. Furthermore, since $\psi\models_{\BDtriangle}\phi\wedge\triangle\phi'$ for every $\psi\in\mathfrak{C}$ s.t.\ $\psi\models_{\BDtriangle}\phi$, we have that $|\phi|^+\cap(\mathsf{C}\setminus|\triangle\phi'|^+)=\varnothing$, whence $\mathscr{M}_\mathfrak{C}\models{\sim}\Be(\phi\wedge\neg\triangle\phi')$, as required. As the converse direction can be proved in the same manner, the result follows.
\end{proof}
\begin{theorem}\label{theorem:presumptivelyvalidrepresentation}
Let $\mathfrak{C}=\langle\mathsf{C},\preceq\rangle$ be a $\BDtriangle$ case model and $\mathscr{M}_{\mathfrak{C}}$ its $\mu$-counterpart. Then the following holds.
\begin{enumerate}
\item $\mathfrak{C}\models\phi\rightsquigarrow^+\phi'$ and $\chi$ is $\langle\phi,\phi'\rangle$'s witnessing case iff
\begin{center}
$\mathscr{M}_\mathfrak{C}\models{\sim\sim}\Be(\phi\wedge\triangle\phi')\wedge\triangle\Be(\chi\Rrightarrow(\phi\wedge\triangle\phi'))\wedge\bigwedge_{\chi'\in\mathsf{C}}\left(\triangle\Be(\chi\Rrightarrow\phi)\rightarrow\triangle(\Be\chi'\rightarrow\Be\chi)\right)$
\end{center}
\item $\mathfrak{C}\models\phi\rightsquigarrow^-\phi'$ and $\chi$ is $\langle\phi,\phi'\rangle$'s witnessing case iff
\begin{center}
$\mathscr{M}_\mathfrak{C}\models{\sim\sim}\Be(\phi\wedge\neg\triangle\neg\phi')\wedge\triangle\Be(\chi'\Rrightarrow(\phi\wedge\neg\triangle\neg 
\phi'))\wedge\bigwedge_{\chi'\in\mathsf{C}}\left(\triangle\Be(\chi'\Rrightarrow\phi)\rightarrow\triangle(\Be\chi'\rightarrow\Be\chi)\right)$
\end{center}
\end{enumerate}
\end{theorem}
\begin{proof}
We prove 1.\ as 2.\ can be proven in the same manner. Assume that $\langle\phi,\phi'\rangle$ is positively presumptively valid over $\mathfrak{C}$ and that $\chi$ is its witnessing case. Then, $\langle\phi,\phi'\rangle$ is positively coherent (whence, $\mathscr{M}\models{\sim\sim}\Be(\phi\wedge\neg\triangle\neg\phi')$) and $\chi\models_{\BDtriangle}\phi\wedge\triangle\phi'$. Thus, $|\chi|^+\subseteq|\phi\wedge\triangle\phi'|^+$ and $|\chi|^-\supseteq|\phi\wedge\triangle\phi'|^-$ \emph{for every model $\mathscr{M}$}, whence $\mathscr{M}\models\triangle\Be(\chi\Rrightarrow(\phi\wedge\neg\triangle\neg 
\phi'))$. Finally, we also have that $\chi'\preceq\chi$ for every $\chi'\in\mathsf{C}$ s.t.\ $\chi\models_{\BDtriangle}\phi$. But this means that for every such $\chi'$, $\mu(|\chi'|^+)\leq\mu(|\chi|^+)$\footnote{Recall that since $\mathsf{c}_i\wedge\mathsf{c}_j\models_{\BDtriangle}\bot$ for every $\mathsf{c}_i,\mathsf{c}_j\in\mathsf{C}$, we have that $\mu(|\mathsf{c}_k|^+)=\mu(\{w_k\})$ for all $\mathsf{c}_k\in\mathsf{C}$.} and thus, $\mathscr{M}\models\triangle(\Be\chi'\rightarrow\Be\chi)$. Hence, $\mathscr{M}\models\bigwedge_{\chi'\in\mathsf{C}}\left(\triangle\Be(\chi'\Rrightarrow\phi)\rightarrow\triangle(\Be\chi'\rightarrow\Be\chi)\right)$, as required.

For the converse, let $\langle\phi,\phi'\rangle$ be \emph{not} positively presumptively valid argument with $\chi$ as the witnessing case. Then at least one of the following holds: (1) $\langle\phi,\phi'\rangle$ is not positively coherent; (2) $\chi\not\models_{\BDtriangle}\phi\wedge\triangle\phi'$; (3) there is some $\chi'\in\mathsf{C}$ s.t.\ $\chi'\models_{\BDtriangle}\phi$ but $\chi'\succ\chi$. Now, for (1), $\mathscr{M}\not\models{\sim\sim}\Be(\phi\wedge\neg\triangle\neg\phi')$; for (2), $\mathscr{M}\not\models\triangle\Be(\chi\Rrightarrow(\phi\wedge\triangle\phi'))$; and finally, for (3), $\mathscr{M}\not\models\bigwedge_{\chi'\in\mathsf{C}}\left(\triangle\Be(\chi'\Rrightarrow\phi)\rightarrow\triangle(\Be\chi'\rightarrow\Be\chi)\right)$.
\end{proof}
\noindent Recall that conclusive and coherent arguments over \emph{classical} case models were represented by means of \emph{conditional probabilities} in~\cite{Verheij2016,Verheij2017}. Here, we did not need conditionalisations on capacities as we used a purely logical representation and could employ $\rightarrow$ in order to ‘simulate’ conditionalised measures.
\section{Conclusion\label{sec:conclusion}}
In this paper, we provided paraconsistent counterparts to the case models discussed in~\cite{Verheij2016,Verheij2017} that use $\BDtriangle$ as their underlying logic. We showed how to recover classical presumptive reasoning from $\BDtriangle$ case models (Theorem~\ref{theorem:BDtrianglecounterparts}). Moreover, we constructed a two-layered logic $\QGBDtriangle$ over $\BDtriangle$ and $\biG$ and used it to establish a representation of arguments with $\QGBDtriangle$ formulas (Theorems~\ref{theorem:coherentrepresenation}--\ref{theorem:presumptivelyvalidrepresentation}).

The natural next steps would be as follows. First, it is instructive to provide a complete axiomatisation of $\QGBDtriangle$. Second, while in this paper we used a linear preference relation (as it is traditionally done, cf., e.g.,~\cite{Roberts1985}), one could argue that if an agent cannot choose between two cases $\mathsf{c}$ and $\mathsf{c}'$, it does not mean that they prefer them to the same degree. It is, hence, reasonable to explore case models whose preference relation is a \emph{partial} preorder. Finally, we managed to represent preference relations on case models as capacities on their $\BDtriangle$ counterparts. An expected question to ask is which properties we have to require from the case model so that its preference relation be represented as a stronger measure: e.g., a belief function, a plausibility function, or a probability measure.
\bibliographystyle{eptcs}
\bibliography{generic}
\newpage
\end{document}

%% file: macros.tex
\newcommand{\coimplies}{\Yleft}
\newcommand{\CPL}{\mathsf{CPL}}
\newcommand{\BD}{\mathsf{BD}}
\newcommand{\BDtriangle}{\mathsf{BD}\triangle}
\newcommand{\LBD}{\mathscr{L}_\mathsf{BD}}
\newcommand{\LBDtriangle}{\mathscr{L}_{\mathsf{BD}\triangle}}
\newcommand{\Prop}{\mathtt{Prop}}

\newcommand{\Be}{\mathtt{B}}
\newcommand{\true}{\mathbf{T}}
\newcommand{\both}{\mathbf{B}}
\newcommand{\neither}{\mathbf{N}}
\newcommand{\false}{\mathbf{F}}
\newcommand{\biG}{\mathsf{biG}}

\newcommand{\triangleMCB}{\mathsf{MCB}_\triangle}

\newcommand{\LtriangleMCB}{\mathscr{L}_{\triangleMCB}}

\newcommand{\QGBDtriangle}{\mathsf{QG}_{\BDtriangle}}
\newcommand{\QG}{\mathsf{QG}}
\newcommand{\LQGBDtriangle}{\mathscr{L}_{\QGBDtriangle}}
\newcommand{\LbiG}{\mathscr{L}_{\biG}}
\newcommand{\four}{\mathbf{4}}